\documentclass[11pt]{amsart}

\usepackage{amsmath}

\usepackage{amsfonts}

\usepackage{amssymb}

\newtheorem{lem}{Lemma}[section]

\newtheorem{teo}[lem]{Theorem}

\newtheorem{cor}[lem]{Corollary}

\newtheorem{defi}[lem]{Definition}

\newcommand{\dive}{\mathrm{div}}

\renewcommand{\phi}{\varphi}

\newcommand{\R}{{\mathbb{R}}}

\newcommand{\bbR}{{\mathbb{R}}}

\newcommand{\average}{{\mathchoice {\kern1ex\vcenter{\hrule height.4pt

width 6pt

depth0pt} \kern-9.7pt} {\kern1ex\vcenter{\hrule height.4pt width 4.3pt

depth0pt}

\kern-7pt} {} {} }}

\title[Rigidity results for
elliptic
boundary value problems]{Rigidity results \\
for elliptic
boundary value problems:\\
stable solutions for quasilinear equations \\
with Neumann or Robin boundary conditions}

\thanks{Part of this work was done while A. P. was
visiting the Dipartimento di Matematica ``Federigo Enriques''
of the Universit\`a di Milano.
The authors are members
of {\em Gruppo Nazionale per l'Analisi
Ma\-te\-ma\-ti\-ca, la Probabilit\`a e
le loro Applicazioni} (GNAMPA) of the
{\em Istituto Nazionale di Alta Matematica} (INdAM). The first and third authors
are supported by the Australian Research Council grant
``N.E.W.'' {\em Nonlocal Equation at Work}.}

\author{Serena Dipierro, Andrea Pinamonti, Enrico Valdinoci}

\address{Dipartimento di Matematica ``Federigo Enriques'',
Universit\`a degli studi di Milano,
Via Saldini 50, 20133 Milano (Italy).}

\email{serena.dipierro@unimi.it}

\address{Dipartimento di Matematica, Universit\`a di Trento,
Via Sommarive 14, 38050 Povo, Trento (Italy).}

\email{Andrea.Pinamonti@gmail.com}

\address{Dipartimento di Matematica ``Federigo Enriques'',
Universit\`a degli studi di Milano,
Via Saldini 50, 20133 Milano (Italy), and
School of Mathematics and Statistics,
University of Melbourne, Grattan Street, 
Parkville, VIC-3010 Melbourne (Australia), and
Istituto di Matematica Applicata e Tecnologie Informatiche,
Via Ferrata 1, 27100 Pavia (Italy).
}

\email{enrico.valdinoci@unimi.it}

\begin{document}

\begin{abstract}
We provide a general approach to the classification results
of stable solutions of (possibly nonlinear) elliptic problems
with Robin conditions.

The method is based on a geometric formula of Poincar\'e type,
which is
inspired by a classical work of Sternberg and
Zumbrun and which
gives an accurate description of the curvatures
of the level sets of the stable solutions. {F}rom this, we show that the stable solutions
of a quasilinear problem with Neumann data are necessarily constant.

As a byproduct of this, we obtain an alternative proof of
a celebrated result of Casten and Holland, and Matano.

In addition, we will obtain as a consequence a new proof of
a result recently established by Bandle, Mastrolia, 
Monticelli and Punzo.
\end{abstract}

\maketitle

\tableofcontents

\section{Introduction}  

The study of stable solutions of variational problems is a classical topic
of investigation. Roughly speaking, stable solutions are critical points
of an energy functional with positive second variation (in particular,
local minimizers of the energy are stable solutions).

Given their special energetic properties, stable solutions often enjoy
better qualitative and quantitative properties than the other solutions
and, in some cases, they can be
completely classified. We refer to the monograph~\cite{DUP} and the references
therein for a complete presentation of the main results available on stable solutions.

The goal of this paper is to provide an approach towards the classification
of stable solutions with Neumann and Robin boundary conditions
which is based on a geometric Poincar\'e formula.
This formula is inspired by the one
originally introduced by Sternberg and
Zumbrun in~\cite{SZ1, SZ2} and provides an accurate bound on the second fundamental
form and on the tangential derivatives of the level sets of any stable solution,
in a form which is remarkably independent from the nonlinearity.
For this reason, we think that the geometric formula obtained
is interesting in itself and reveals important information
on the level sets of the solutions.

With this approach, we obtain a classification result for (possibly
nonlinear) elliptic problems
with Neumann boundary conditions
(these operators can be also singular or degenerate,
and we comprise the cases of the $p$-Laplacian and of the mean curvature equation). This result contains, as a particular case,
the classical one obtained by Casten and Holland in~\cite{CH} and by Matano in~\cite{Mat}, for which
we produce a different proof (in fact, with weaker regularity assumptions).

In addition, the method also gives a classification result
that has been recently obtained by Bandle, Mastrolia,
Monticelli and Punzo in~\cite{punzo}, see also \cite{DPV}.\medskip

The mathematical framework in which we work is the following.
Let $\Omega\subset \bbR^n$ be a bounded domain with smooth boundary.
We study the solutions to the following boundary value problem:
\begin{align}\label{eqncomp}
\bigg \{
\begin{array}{rl}
\dive(a(|\nabla u|)\nabla u)+f(u)=0& \mbox{in}\ \Omega, \\
a(|\nabla u|)\partial_{\nu}u+h(u)=0 & \mbox{on}\ \partial\Omega, \\
\end{array}
\end{align} 
where $f $, $h\in C^{1}(\bbR)$
and~$a\in C^{1,1}_{loc}((0,+\infty))$.
Here~$\nu$ denotes the external unit normal to~$\partial\Omega$.
We also require that the function $a$ satisfies 
the following structural conditions:
\begin{align}\label{cond1}
&a(t)>0\quad \mbox{for any}\ t\in (0,+\infty),\\
\label{cond2}
&a(t)+a'(t)t>0\quad \mbox{for any}\ t\in (0,+\infty).
\end{align}
We observe that the general form of \eqref{eqncomp}
and the structural conditions~\eqref{cond1} and~\eqref{cond2}
encompass, as very special cases,
many elliptic singular and degenerate equations.
Indeed, if~$a(t):=t^{p-2}$, with~$1<p<+\infty$,
or~$a(t):=1/ \sqrt{1+t^2}$, then we obtain the
$p$-Laplacian and the mean curvature equations respectively.

Following \cite{FarHab, FSV, CNP}, we
define~$A:\bbR^n\to Mat(n\times n)$
and~$\lambda_1\in C^0((0,+\infty))$ as follows
\begin{align}\label{DefA}
A_{hk}(\xi):=\frac{a'(|\xi|)}{|\xi|}\xi_h\xi_k+a(|\xi|)\delta_{hk}\quad \mbox{for any}\  1\leq h,k\leq n,
\end{align}
\begin{align}\label{defa}
\lambda_1(t):=a(t)+a'(t)t\quad \mbox{for any}\  t>0.
\end{align}

\begin{defi} We say that $u$ is a weak solution to~\eqref{eqncomp}
if $u\in C^1(\overline{\Omega})$,
\begin{align}\label{weak}
\int_{\Omega}a(|\nabla u|)\left\langle \nabla u,
\nabla\varphi\right\rangle \,dx+\int_{\partial\Omega}
 h(u)\, \varphi\,d\sigma
-\int_{\Omega} f(u)\varphi\,dx =0,
\end{align}
for any~$\varphi\in C^1(\overline{\Omega})$,
and either (A1) or (A2) is satisfied, where
\begin{itemize}
\item[(A1)] $\{\nabla u=0\}=\varnothing$;
\item[(A2)]  $a\in C^0([0,+\infty))$
and
\[
\mbox{the map}\quad t\to ta(t)\quad \mbox{belongs to}\quad C^1([0,+\infty)).
\]
\end{itemize}
\end{defi}

\medskip

Notice that the first integrand in~\eqref{weak} is well-defined, thanks to either (A1) or (A2)
(recall the boundary condition in~\eqref{eqncomp}). 

\smallskip

We observe that, in general, 
equation~\eqref{eqncomp} may have no solution. For instance,
in the mean curvature equation case in which~$a(t) := \frac1{\sqrt{1+t^2}}$,
a compatibility condition was discovered in~\cite{GI},
see also Theorem~1.3 and Remark~1.4 in~\cite{MWW}: in
particular, if~$f$ is a nonzero constant and~$h:=0$, there is no solution.
\smallskip

The regularity assumption $u\in C^1(\overline{\Omega})$
is always fulfilled in many important cases,
like those involving the
$p$-Laplacian operator or the mean curvature
operator (see e.g.~\cite{Lieb}).

In light of this, and in view of the great generality
of the functions~$a$ and~$h$,
it is natural to work in the above setting.

\begin{defi}\label{phstab}
Let $u$ be a weak solution to~\eqref{eqncomp}.
We say that~$u$ is stable if
\begin{align}\label{hstab}
\int_{\Omega}\left\langle
A(\nabla u)\nabla\varphi,\nabla\varphi
\right\rangle\,dx+
\int_{\partial\Omega}h'(u)\,\varphi^2\,d\sigma
-\int_{\Omega} f'(u)\varphi^2\,dx\ge 0,
\end{align}
for any~$\varphi\in C^{1}(\overline{\Omega})$.
\end{defi} 

We notice that, as customary, from the variational point of view,
such definition of stability is equivalent to the fact that the associated
energy functional is nonnegative defined (nevertheless, we do not need
to explicitly introduce such energy
setting, since the framework provided
by equation~\eqref{eqncomp}, combined with the condition in~\eqref{hstab},
is sufficient for our purposes).\medskip

In the subsequent formula~\eqref{prepoi2}, we give
an extension of the geometric formula obtained
in~\cite{SZ1, SZ2}. Such formula relates
the stability of the equation with the
principal curvatures of the corresponding
level set and with the tangential
gradient of the solution. Since this formula
bounds a weighted~$L^2$-norm of
any~$\varphi\in C^1(\overline{\Omega})$ plus a boundary
term by a weighted~$L^2$-norm of its gradient,
we may see it as a weighted Poincar\'e type inequality.

The first result towards a geometric Poincar\'e inequality
is the following:

\begin{teo}\label{TH:P}
Let $u\in C^2(\overline{\Omega})$ be a stable weak solution
to~\eqref{eqncomp}. 
Then, for every~$\varphi\in C^1(\overline{\Omega})$, it holds
\begin{equation}\begin{split}\label{15imp}
&\int_{\{\nabla u\neq 0\}\cap \Omega}
\Big[\sum_{i=1}^n\left\langle
A(\nabla u)\nabla u_i, \nabla u_i \right\rangle-\left\langle A(\nabla u)
\nabla|\nabla u|,
\nabla|\nabla u|\right\rangle
\Big]\varphi^2\, dx\\&\qquad
+\int_{\partial\Omega} \Big(f(u)\,\partial_\nu u
- a(|\nabla u|)\left\langle \nabla u, \partial_{\nu}
\nabla u \right\rangle
-h(u)\,\Delta u-h'(u)|\nabla u|^2\Big)\varphi^2\,d\sigma
\\
\leq\;& \int_{\Omega}|\nabla u|^2\left\langle A(\nabla u)
\nabla\varphi,\nabla\varphi\right\rangle\, dx
.
\end{split}\end{equation}
\end{teo}

Concerning the regularity assumption on the solution
taken in Theorem~\ref{TH:P}, we think that it is an interesting problem to determine whether
similar results can be obtained under weaker regularity. This would be particularly
interesting in the case of $p$-Laplace-type equations, in which the gradient of the solution
is usually not better than H\"older continuous at critical points. On the one hand,
in our arguments, assuming the continuity of the second derivatives
makes the computations available ``up to the boundary'' and gives perfect 
sense of
the terms~$
a(|\nabla u|)\left\langle \nabla u, \partial_{\nu}
\nabla u \right\rangle$ and~$h(u)\,\Delta u$ along~$\partial\Omega$.
On the other hand, it is possible that a more careful analysis provides
a suitable meaning for these terms also in a less regular situation:
for instance, at points where~$h(u)\ne0$, the Robin boundary condition suggests
that~$\nabla u\ne0$, which makes the elliptic regularity available (thus
reducing to the smooth case), while at points where~$h(u)=0$
the term~$h(u)\,\Delta u$ formally disappears. Similarly, at critical points,
the term~$a(|\nabla u|)\left\langle \nabla u, \partial_{\nu}
\nabla u \right\rangle$ may disappear under weaker regularity assumptions
than those requested in
Theorem~\ref{TH:P},
or have a useful sign at least in the case of convex domains (see Lemma~\ref{LemConvex}).
These observations indeed suggest that the results presented in this paper
may be valid in further generality.
\medskip

We also observe that the integrand in the first line
of~\eqref{15imp}
has indeed a geometric interpretation in terms of the
curvatures of the level sets
of~$u$ and the tangential gradient.

For this, given a point~$x\in\R^n$, we let
\[
L_{u,x}:=\{y\in\bbR^n {\mbox{ s.t. }} u(y)=u(x)\}
\]
be the level set of~$u$ passing through~$x$.
Furthermore,
we denote by~$\nabla_{T}u$ the tangential gradient
of~$u$ along~$L_{u,x}\cap \{\nabla u\neq 0\}$,
and by~$k_1,\ldots, k_{n-1}$
the principal curvatures of~$L_{u,x}\cap \{\nabla u\neq 0\}$.

With this notation, we have the following:

\begin{cor}\label{COR:P}
Let $u\in C^2(\overline{\Omega})$ be a stable weak solution
to~\eqref{eqncomp}. 
Then, for every~$\varphi\in C^1(\overline{\Omega})$, it holds
\begin{equation}\begin{split}\label{prepoi2}
&\int_{\Omega}
\Big[\lambda_1|\nabla_T|\nabla u||^2
+a(|\nabla u|)|\nabla u|^2 \sum_{j=1}^{n-1} k_j^2\Big]\varphi^2\,dx
\\&\qquad
+\int_{\partial\Omega} \Big(f(u)\,\partial_\nu u
- a(|\nabla u|)\left\langle \nabla u, \partial_{\nu}
\nabla u \right\rangle
-h(u)\,\Delta u-h'(u)|\nabla u|^2\Big)\varphi^2\,d\sigma
\\
\leq\;& \int_{\Omega}|\nabla u|^2\left\langle A(\nabla u)
\nabla\varphi,\nabla\varphi\right\rangle\, dx
.\end{split}\end{equation}
\end{cor}

As already mentioned, this type of formulas 
has been originally introduced by Sternberg and
Zumbrun in~\cite{SZ1, SZ2}, and then used to prove symmetry
and rigidity results in~\cite{FarHab, FSV, {ASV2}}.

Since then, this type of inequalities has been 
applied in several contexts:
to prove rigidity results for boundary reaction-diffusion equations,
see~\cite{sire1,sire2} (where a celebrated conjecture of De Giorgi
for equations driven by the fractional Laplacian is also proved
in dimension~2),
to handle
semilinear equations in Riemannian and Sub-Riemannian spaces,
see~\cite{FMV,fsv2,fsv1,FP,FV1,PV},
to study problems involving the Ornstein-Uhlenbeck operator,
see~\cite{CNV},
as well as semilinear equations with unbounded drift~\cite{CNP,FNP}
and systems of PDEs~\cite{fazly, DI, DP1,DP2}. 

Recently, in~\cite{DSV},
the case of Neumann boundary condition for boundary
reaction-diffusion equations was dealt with the use
of a Poincar\'e inequality involving also a boundary term. 
\medskip

Now, we present some rigidity results, in the spirit of~\cite{CH},
by taking advantage of the geometric information
given by 
formula~\eqref{prepoi2}. 
\medskip

We first deal with the case of the Neumann boundary
condition, that is we choose~$h:=0$ in~\eqref{eqncomp}.
In this setting we get the following result:

\begin{teo}\label{main1}
Let $\Omega\subset \bbR^n$ be a bounded domain
with smooth boundary. Suppose that~$\Omega$ is convex,
with strictly positive principal
curvatures along~$\partial\Omega$.

Let $u\in C^2(\overline{\Omega})$ be a stable weak
solution to~\eqref{eqncomp} with~$h:=0$. Then~$u$
is constant in $\Omega$.
\end{teo}

Theorem~\ref{main1} is new even in the case of the $p$-Laplacian,
but our proof is robust enough to deal with a very general class
of operators. 

Also, we point out that Theorem~\ref{main1}
provides, as a byproduct, a new proof of Theorem~1
in~\cite{CH} (see also~\cite{Mat})
when the function~$a:=1$ and so the general
operator in~\eqref{eqncomp} boils down to the Laplacian.
\medskip

As a corollary of Theorem~\ref{TH:P}, we also 
give an alternative proof 
of Theorem~2.1 in~\cite{punzo} (see also~\cite{jimbo}). More precisely, 
we deal with the case of the Laplacian in dimension~2
with Robin boundary condition, and we obtain:

\begin{teo}\label{main2}
Let $\Omega\subset \bbR^2$ be a bounded domain
with smooth boundary. 
Let $u\in C^2(\overline{\Omega})$ be a
solution to~\eqref{eqncomp} with~$a:=1$
and~$h(u):=\alpha u$, for some~$\alpha\in\R$.

Suppose that
\begin{equation}\begin{split}\label{vediamo}
\int_{\partial\Omega}
\alpha^2u^2\Big(
\frac{f(u)}{\alpha u}-\kappa +\alpha\Big)\,d\sigma<0\qquad
{\mbox{and }}\qquad \alpha+\kappa\ge 0\; {\mbox{ on }}\partial\Omega,
\end{split}
\end{equation}
where~$\kappa$ is the curvature of~$\partial\Omega$.

Then~$u$ is unstable.
\end{teo}

It would be desirable to grasp a full understanding of the stable solutions
of general quasilinear equations and of the role played by the geometry of the domain
(with respect to this, it is likely that the convexity assumption can be relaxed,
see e.g. the comments in the Remark before formula~(13) in~\cite{CH}).
To keep in mind some examples, we observe that:
\begin{itemize}
\item The convexity assumption in Theorem~\ref{main1} cannot be completely removed.
Indeed, suitable ``dumbbell'' domains
which produce nonconstant stable solutions of~\eqref{eqncomp}
have been constructed in Section~6 of~\cite{Mat} (see in particular the figure
on page~452 of~\cite{Mat} and Remark~6.4 on page~453 of~\cite{Mat}).
The original examples of~\cite{Mat} took into account
connected regions with quite complex shapes, and simpler
examples have been investigated
also in~\cite{HAVE, VEGA, J1, J2, J3, J4, OLI, ARR, BOLI}.
\item If~$\R\ni x\mapsto u(x)$ is a smooth function such that~$u(x)=-1$
if~$x\le-1$, $u(x)=1$
if~$x\ge1$ and~$0\le u'\le 4$, and if~$a(t)=0$ if~$t\in[0,4]$, then~$u$ is a solution
of~\eqref{eqncomp} in $\Omega:=(-2,2)$
with~$h:=0$ and~$f:=0$, which is not constant. This says that the nondegeneracy of the coefficient~$a$
cannot be completely removed in the statement of Theorem~\ref{main1}.
\item The thesis in Theorem~\ref{main2} is quite strong, since
it also excludes constant solutions. This is due to
condition~\eqref{vediamo}. For instance,
the function vanishing identically is a stable solution of~\eqref{eqncomp}
with~$a:=1$, $f:=0$ and~$h(u):=\alpha u$,
for any~$\alpha\ge0$. But this solution does not verify the first condition in~\eqref{vediamo}.
\end{itemize}
\medskip

The paper is organized as follows. In Section~\ref{sec1}
we prove Theorem~\ref{TH:P}
and Corollary~\ref{COR:P}. 
Then, Sections~\ref{sec2} and~\ref{sec3} contain
the proofs of Theorems~\ref{main1} and~\ref{main2}.

\section{A geometric Poincar\'e inequality: proofs
of Theorem~\ref{TH:P}
and Corollary~\ref{COR:P}}\label{sec1}

In this section we deal with the Poincar\'e-type inequality
and we give the proof of Theorem~\ref{TH:P}
and Corollary~\ref{COR:P}.
\medskip

We start recalling the following result,
which has been proved in~\cite{FSV}.

\begin{lem}
For any $\xi\in \bbR^n\setminus \{0\}$, the matrix $A(\xi)$
is symmetric and positive definite,
and its eigenvalues are~$\lambda_1(|\xi|),\cdots,
\lambda_n(|\xi|)$, where~$\lambda_1$ is as in~\eqref{defa}
and~$\lambda_i(t):=a(t)$ for every $i=2,\ldots, n$. Moreover,
\begin{align*}
\left\langle A(\xi)\xi,\xi\right\rangle=|\xi|^2\lambda_1(|\xi|),
\end{align*}
and
\begin{align*}
\left\langle A(\xi)(V-W),(V-W)\right\rangle=\left\langle A(\xi)V,V\right\rangle+\left\langle A(\xi)W,W\right\rangle-2\left\langle A(\xi)V,W \right\rangle,
\end{align*}
for any $V,W\in\bbR^n$ and any $\xi \in \bbR^n\setminus\{0\}$.
\end{lem}

The forthcoming formula~\eqref{eqngen} is a fundamental step
towards the proof of Theorem~\ref{TH:P}.
We let~$\nu=(\nu_1,\ldots,\nu_n)$ be the unit
normal to~$\partial\Omega$.

\begin{lem}\label{eqnlin}
Let $u\in C^2(\overline{\Omega})$ be a weak solution to~\eqref{eqncomp}.
Then, for any $i=1,\ldots, n$,
and any $\varphi\in C^1(\overline{\Omega})$, we have
\begin{equation}\begin{split}\label{eqngen}
&\int_{\Omega} \left\langle A(\nabla u)\nabla u_i,
\nabla\varphi \right\rangle \,dx-
\int_{\Omega} f'(u) u_i \varphi\, dx
\\=\;&\int_{\partial\Omega} a(|\nabla u|)\left\langle \nabla u,
\nabla\varphi\right\rangle \nu_i\,  d\sigma-
\int_{\partial\Omega} f(u)\varphi \nu_i \, d\sigma
+\int_{\partial\Omega}h(u)\,\varphi_i\,d\sigma.
\end{split}\end{equation}
\end{lem}

\begin{proof}
By Lemma $2.2$ in \cite{FSV} we have that
\begin{align*}
\mbox{the map}\quad x\to W(x):=a(|\nabla u(x)|)\nabla u(x)\quad \mbox{belongs to}\ W^{1,1}_{loc}(\Omega,\bbR^n),
\end{align*}
By Stampacchia's Theorem (see e.g. \cite[Theorem 6:19]{LL}),
we get that~$\partial_i W=0$ for almost any~$x\in \{W=0\}$.
In the same way, by Stampacchia's Theorem and (A2),
it can be proven that $\nabla u_i(x)=0$, and hence $A(\nabla u(x))\nabla u_i(x)=0$, for almost any $x\in \{\nabla u=0\}\cap \Omega$.
Moreover, 
\begin{align}\label{wq2}
	\partial_i W=A(\nabla u)\nabla u_i  \quad a.e.\ \mbox{in}\ \Omega.
\end{align}
Applying~\eqref{weak} with $\varphi$
replaced by~$\varphi_i$, making use of~\eqref{wq2}
and the integration by parts formula, we get
\begin{eqnarray*}
0&=&\int_{\Omega}a(|\nabla u|)\left\langle
\nabla u,\nabla\varphi_i\right\rangle\,dx
+\int_{\partial\Omega}h(u)\,\varphi_i\,d\sigma
-\int_{\Omega}f(u)\,\varphi_i\, dx\\
&=& \int_\Omega\left[ \partial_i\Big( 
a(|\nabla u|)\left\langle
\nabla u,\nabla\varphi\right\rangle\Big) -
\left\langle A(\nabla u)\nabla u_i,
\nabla\varphi\right\rangle\right]\, dx\\
&&\qquad +\int_{\partial\Omega}h(u)\,\varphi_i\,d\sigma
-\int_\Omega \left[\partial_i\big( f(u)\,\varphi
\big)- f'(u)\,u_i\,\varphi \right]\, dx\\
&=&-\int_{\Omega}\Big[\left\langle A(\nabla u)\nabla u_i,
\nabla\varphi\right\rangle-f'(u)\,u_i\,\varphi \Big]\, dx\\
&&\qquad +\int_{\partial\Omega} a(|\nabla u|)
\left\langle \nabla u,\nabla\varphi\right\rangle \nu_i\, d\sigma
-\int_{\partial\Omega} f(u)\,\varphi \,\nu_i\, d\sigma
+\int_{\partial\Omega}h(u)\,\varphi_i\,d\sigma,
\end{eqnarray*}
which proves~\eqref{eqngen}.
\end{proof}

Concerning the proof of Lemma~\ref{eqnlin}
and the fact that we separate the analysis of the region in which
the gradient vanishes with respect to the one in which the gradient
differs from zero, we remark that, in the generality considered here,
it is possible that both of these sets
have positive Lebesgue measure simultaneously. See for instance the examples
in Propositions~7.2 and~7.3 in~\cite{FSV}
dealing with one-dimensional $p$-Laplace equations with~$p>2$.
\smallskip

{F}rom now on, we use $A$ and $a$,
as a short-hand notation
for~$A(\nabla u)$ and~$a(|\nabla u|)$,
respectively. 

We now provide the proof of Theorem~\ref{TH:P}:

\begin{proof}[Proof of Theorem~\ref{TH:P}]
We start by observing that by Stampacchia's Theorem we get
\begin{eqnarray}\label{16}
&\nabla|\nabla u|(x)=0\quad \mbox{a.e.}\ x\in \{|\nabla u|=0\},\\
\label{17}
&\nabla u_j(x)=0\quad \mbox{a.e.}\ x\in \{|\nabla u|=0\}\subseteq \{u_j=0\} \quad \mbox{for any} \ j=1,\ldots, n.
\end{eqnarray}
Now, let $\varphi\in C^1(\overline{\Omega})$
and $i=1,\ldots,n$.
Using~\eqref{eqngen} with test function~$u_i\varphi^2$, we obtain that
\begin{eqnarray*} 
&& \int_{\Omega} \left\langle A\nabla u_i,
\nabla(u_i\varphi^2) \right\rangle \,dx-
\int_{\Omega} f'(u) u_i^2 \varphi^2\, dx
\\&=&\int_{\partial\Omega} a\left\langle \nabla u,
\nabla(u_i\varphi^2) \right\rangle \nu_i\,  d\sigma-
\int_{\partial\Omega} f(u)u_i\varphi^2 \nu_i \, d\sigma
+\int_{\partial\Omega}h(u)\,\partial_i(u_i\varphi^2)\,d\sigma\\
&=&\int_{\partial\Omega} a\left\langle \nabla u,
\nabla(u_i\varphi^2) \right\rangle \nu_i\,  d\sigma-
\int_{\partial\Omega} f(u)u_i\varphi^2 \nu_i \, d\sigma
+\int_{\partial\Omega}h(u)\big(u_{ii}\,\varphi^2
+ 2u_i\varphi\varphi_i\big)\,d\sigma.
\end{eqnarray*}
Hence, summing over $i=1,\ldots, n$ and recalling~\eqref{16}
and~\eqref{17}, we get
\begin{equation*}\begin{split}
&
\int_{\{\nabla u\neq 0\}\cap\Omega }\left(
\sum_{i=1}^n \left\langle A\nabla u_i,
\nabla u_i\right\rangle\varphi^2+2\varphi\, |\nabla u|
\left\langle A\nabla \varphi,
\nabla|\nabla u| \right\rangle\right)\,dx -\int_\Omega
f'(u)|\nabla u|^2\varphi^2 \, dx\\
=\;&
\int_{\Omega}\left(\sum_{i=1}^n \left\langle A\nabla u_i,
\nabla(u_i\varphi^2) 
\right\rangle -f'(u)|\nabla u|^2\varphi^2\right)  \, dx \\
=\;&\int_{\partial\Omega}\sum_{i=1}^n a
\left\langle \nabla u, \nabla(u_i\varphi^2)\right\rangle 
\nu_i \,d\sigma-\int_{\partial\Omega} f(u) \varphi^2 
\left\langle \nabla u,\nu\right\rangle\, d\sigma
+\int_{\partial\Omega}h(u)\big(\Delta u\,\varphi^2
+ \nabla u\cdot\nabla(\varphi^2)\big)\,d\sigma\\
=\;& 
\int_{\partial\Omega} a\left\langle \nabla u,
\partial_{\nu} \nabla u \right\rangle\varphi^2 \, d\sigma
+ \int_{\partial\Omega} a\left\langle \nabla u,
\nabla (\varphi^2) \right\rangle \partial_\nu u\,d\sigma
-\int_{\partial\Omega} f(u) \varphi^2 
\left\langle \nabla u,\nu\right\rangle \,d\sigma\\&\qquad 
+\int_{\partial\Omega}h(u)\Big(\Delta u\,\varphi^2
+\left\langle
\nabla u,\nabla(\varphi^2)\right\rangle\Big)\,d\sigma.
\end{split}\end{equation*}
Now we recall the Robin condition
in~\eqref{eqncomp} and we obtain that
\begin{equation}\begin{split}\label{we4}&
\int_{\{\nabla u\neq 0\}\cap\Omega }\left(
\sum_{i=1}^n \left\langle A\nabla u_i,
\nabla u_i\right\rangle\varphi^2+2\varphi\, |\nabla u|
\left\langle A\nabla \varphi,
\nabla|\nabla u| \right\rangle\right)\,dx -\int_\Omega
f'(u)|\nabla u|^2\varphi^2 \, dx\\
=\;& 
\int_{\partial\Omega} a\left\langle \nabla u,
\partial_{\nu} \nabla u \right\rangle\varphi^2 \, d\sigma
-\int_{\partial\Omega} h(u)\left\langle \nabla u,
\nabla (\varphi^2) \right\rangle \,d\sigma
-\int_{\partial\Omega} f(u)\,\partial_\nu u \,\varphi^2 
\,d\sigma\\&\qquad 
+\int_{\partial\Omega}h(u)\Big(\Delta u\,\varphi^2
+ \left\langle
\nabla u,\nabla(\varphi^2)\right\rangle\Big)\,d\sigma\\
=\;& 
\int_{\partial\Omega} a\left\langle \nabla u,
\partial_{\nu} \nabla u \right\rangle\varphi^2 \, d\sigma
- \int_{\partial\Omega} f(u)\,\partial_\nu u\,
\varphi^2 \,d\sigma
+\int_{\partial\Omega} h(u)\,\Delta u\,\varphi^2\,d\sigma.
\end{split}\end{equation}

On the other hand, using~\eqref{hstab}
with test function~$|\nabla u|\varphi$
and recalling again~\eqref{16}, we then get
\begin{equation}\begin{split}\label{19}
 0\leq\;& \int_{\Omega}\left(\left\langle 
A\nabla(|\nabla u|\varphi),\nabla(|\nabla u|\varphi)
\right\rangle
-f'(u)|\nabla u|^2\varphi^2\right)\, dx
+\int_{\partial\Omega} h'(u)|\nabla u|^2\varphi^2\,d\sigma\\
=\;&\int_{\Omega}|\nabla u|^2\left\langle A\nabla
\varphi,\nabla\varphi\right\rangle\,dx\\&\qquad
+\int_{\{\nabla u\neq 0\}\cap \Omega}\Big(\left\langle A
\nabla|\nabla u|,\nabla|\nabla u|\right\rangle \varphi^2
+
\left\langle A\nabla\varphi,\nabla|\nabla u|\right\rangle 2\varphi|\nabla u|
\Big)\,dx\\ &\qquad -
\int_{\{\nabla u\neq 0\}\cap\Omega}
f'(u)|\nabla u|^2\varphi^2\, dx
+\int_{\partial\Omega} h'(u)|\nabla u|^2\varphi^2\,d\sigma.
\end{split}\end{equation}
Putting together~\eqref{we4} and~\eqref{19}, we conclude that
\begin{equation*}\begin{split}
 0&\leq \int_{\Omega}|\nabla u|^2\left\langle A
\nabla\varphi,\nabla\varphi\right\rangle\, dx\\
&\qquad+\int_{\{\nabla u\neq 0\}\cap \Omega}
\Big[\left\langle A\nabla|\nabla u|,
\nabla|\nabla u|\right\rangle
-\sum_{i=1}^n\left\langle
A\nabla u_i, \nabla u_i \right\rangle\Big]\varphi^2\, dx\\
&\qquad+\int_{\partial\Omega} a\left\langle \nabla u, \partial_{\nu}
\nabla u \right\rangle\varphi^2\, d\sigma
-\int_{\partial\Omega} f(u)\,\partial_\nu u\,\varphi^2\,
d\sigma\\&\qquad
+\int_{\partial\Omega}h(u)\,\Delta u\,\varphi^2\,d\sigma
+\int_{\partial\Omega}h'(u)|\nabla u|^2\varphi^2\,d\sigma
,
\end{split}\end{equation*}
which is the thesis.
\end{proof}

We complete this section by proving Corollary~\ref{COR:P}.

\begin{proof}[Proof of Corollary~\ref{COR:P}]
By Lemma~2.3 in~\cite{FSV} we see that  
\begin{equation}\begin{split}\label{ugu}
&\left\langle A\nabla|\nabla u|,\nabla|\nabla u|\right\rangle-
\sum_{i=1}^n\left\langle A\nabla u_i,
\nabla u_i \right\rangle\\=\;&a
\Big[|\nabla |\nabla u||^2-\sum_{i=1}^n|\nabla u_i|^2\Big]-
a'|\nabla u||\nabla_T|\nabla u||^2.
\end{split}\end{equation}
Therefore, using~\eqref{defa} we get
\begin{align}\label{ugu2}
&\left\langle A\nabla|\nabla u|,\nabla|\nabla u|\right\rangle-\sum_{i=1}^n\left\langle A\nabla u_i, \nabla u_i \right\rangle\\
\nonumber
&=-\lambda_1|\nabla_T|\nabla u||^2-a(|\nabla u|)\Big(\sum_{i=1}^n|\nabla u_i|^2-|\nabla_T|\nabla u||^2-|\nabla|\nabla u||^2\Big).
\end{align}
Notice that the quantity
$$\sum_{i=1}^n|\nabla u_i|^2-|\nabla_T |\nabla u||^2
-|\nabla|\nabla u||^2 $$  has a geometric interpretation,
in the sense that it can be expressed in terms of the
principal curvatures of level sets of~$u$.
Indeed, the following formula holds (see \cite{FSV,SZ1,SZ2})
\begin{align}\label{eqfund}
	\sum_{i=1}^n|\nabla u_i|^2-|\nabla |\nabla u||^2-|\nabla_T|\nabla u||^2=|\nabla u|^2 \sum_{j=1}^{n-1} k_j^2
	\qquad {\rm on\ }L_{u,x}\cap \{\nabla u\neq 0\}.
\end{align}
With this, formula~\eqref{15imp} becomes
\begin{equation*}\begin{split}
&\int_{\Omega}
\Big[\lambda_1|\nabla_T|\nabla u||^2
+a|\nabla u|^2 \sum_{j=1}^{n-1} k_j^2\Big]\varphi^2\,dx
\\&\qquad
+\int_{\partial\Omega} \Big(f(u)\,\partial_\nu u
- a\left\langle \nabla u, \partial_{\nu}
\nabla u \right\rangle
-h(u)\,\Delta u-h'(u)|\nabla u|^2\Big)\varphi^2\,d\sigma
\\
\leq\;& \int_{\Omega}|\nabla u|^2\left\langle A
\nabla\varphi,\nabla\varphi\right\rangle\, dx
,\end{split}\end{equation*}
which is the desired inequality.
\end{proof} 

\section{Proof of Theorem \ref{main1}}\label{sec2}

In this section we will use formula~\eqref{prepoi2} to prove
Theorem~\ref{main1}, following the approach
introduced in~\cite{FarHab} and then developed in~\cite{FSV,DSV}.

We start with the following result:

\begin{lem}\label{LemConvex}
Let $\Omega\subset \bbR^n$ be an open convex
set with boundary of class $C^2$ and
let~$a\in C^{1,1}_{loc}((0,+\infty))$ satisfying~\eqref{cond1}.
Let $u\in C^2(\overline{\Omega})$, with $\partial_{\nu} u=0$
on $\partial\Omega$.

Then, at each point~$x\in \partial\Omega$ it holds
\begin{align}
a(|\nabla u(x)|)\left\langle \nabla u(x), \partial_{\nu}\nabla u(x)\right\rangle \leq 0.
\end{align}
\end{lem}

\begin{proof}
If $\nabla u(x)\neq 0$ then the thesis follows as in \cite[Theorem 2]{CH} (see also \cite[Lemma 2.1]{DSV}). If $\nabla u(x)=0$ then $\left\langle \nabla u(x), \partial_{\nu}\nabla u(x)\right\rangle= 0$ and the thesis follows as well.
\end{proof}

We also need the following result, whose proof is 
similar to that of Corollary~2.6 in~\cite{DSV} and so is omitted.

\begin{lem}\label{lemn2}
Let~$x_o\in\Omega$, with~$\nabla u(x_o)\neq 0$.
Suppose that
\begin{equation}\label{spero}\left\langle A\nabla|\nabla u|,
\nabla |\nabla u| \right\rangle 
-\sum_{i=1}^n\left\langle A\nabla u_i,
\nabla u_i \right\rangle =0.\end{equation}
Then, each connected component of the level sets of~$u$
must be an~$(n-2)$-dimensional hyperplane
intersected~$\Omega$.
\end{lem}

With these two results, we can
now complete the proof of Theorem~\ref{main1}.

\begin{proof}[Proof of Theorem~\ref{main1}]
The proof is based on the choice of the test function
in~\eqref{prepoi2}. Indeed, 
we let~$\varphi\equiv 1$ in~\eqref{prepoi2}, we recall
the Neumann boundary condition in~\eqref{eqncomp}
and we obtain that
\begin{equation}\begin{split} &\int_{\Omega}
\Big[\lambda_1|\nabla_T|\nabla u||^2
+a|\nabla u|^2 \sum_{j=1}^{n-1} k_j^2\Big]\varphi^2\,dx
-\int_{\partial\Omega} a\left\langle \nabla u, \partial_{\nu}
\nabla u \right\rangle\varphi^2\,d\sigma
\\
& \qquad \le\int_{\Omega}|\nabla u|^2\left\langle A
\nabla\varphi,\nabla\varphi\right\rangle\, dx =0.
\end{split}\end{equation}
Then, recalling Lemma \ref{LemConvex}, we conclude that
\begin{align}\label{trentuno}
	k_j(x)=0,\quad |\nabla_T|\nabla u||(x)=0
\end{align}
for every $j=1,\ldots, n-1$ and
every~$x\in\{\nabla u\neq 0\}$, and
\begin{align}\label{bordo}
\left\langle \nabla u, \partial_{\nu}
\nabla u \right\rangle =0 \quad {\mbox{ on }} \partial\Omega.
\end{align}
Having proved~\eqref{trentuno} and~\eqref{bordo},
the thesis follows as in the proof of Theorem~1.2 in~\cite{DSV}. 
We give the details here for the sake of completeness. 

We first claim that
\begin{equation}\label{JU:tp}
{\mbox{$u$ is constant along~$\partial\Omega$.}}
\end{equation}
To check this, we argue towards a contradiction,
supposing that there exist~$x$, $y\in\partial\Omega$ such
that~$u(x)\ne u(y)$.
{F}rom Lemma~2.2 in~\cite{DSV},
we know that we can connect $x$ and~$y$
with a continuous path~$\sigma:[0,1]\to\partial\Omega$.
So, we define~$\zeta(t):= u(\sigma(t))$, and we see
that~$\zeta(0) \neq \zeta(1)$. 
Hence, 
we can find~$\bar t\in (0,1)$ such that~$\dot\zeta(\bar t)\ne0$, 
namely
\begin{equation} \label{OUGFh} 0\ne
\dot\zeta(\bar t) = \nabla u(\sigma(\bar t)) 
\cdot \dot\sigma(\bar t).\end{equation}
We also let~$\bar x:=\sigma(\bar t)$.
Up to a change of coordinates, we may suppose that the exterior
normal of~$\partial\Omega$ at~$\bar x$
coincides with~$-e_n$, and therefore
$\Omega$ can
be written in normal coordinates as the epigraph
of a function~$\gamma\in
C^2(\R^{n-1})$
near~$\bar x$. 

Since, by assumptions, the principal curvatures of~$\partial\Omega$
are positive, we have that
\begin{equation}\label{PO:k67YY}
{\mbox{the Hessian of~$\gamma$
is positive definite}}.\end{equation}
On the other hand, by~\eqref{bordo} here
and formula~(2.1) in~\cite{DSV},
we have that
\begin{equation}\label{jgwglr94u}
0=-\left\langle \nabla u(\bar x), \partial_{\nu} \nabla u(\bar x) \right\rangle 
=  \sum_{i,j=1}^{n-1}\gamma_{ij}
(\bar x')\,u_{i}(\bar x)\,u_{j}(\bar x).\end{equation}
This and~\eqref{PO:k67YY} imply that~$u_{i}(\bar x)=0$
for any~$i=1,\dots,n-1$. 

Furthermore, we have that~$u_{n}(\bar x)=-\partial_\nu
u(\bar x)=0$,
thanks to the choice of the coordinate system and the Neumann
condition. This and~\eqref{jgwglr94u} give
that~$\nabla u(\bar x)=0$,
in contradiction with~\eqref{OUGFh}, which
proves~\eqref{JU:tp}.

Now, we let~$c$ be the value
attained by~$u$ along~$\partial\Omega$,
as given by~\eqref{JU:tp}, and 
we complete the proof of Theorem~\ref{main1}, by showing that
\begin{equation}\label{JU:tp:2}
{\mbox{$u$ is constant in $\Omega$.}}\end{equation}
Indeed, if this was not true, then we would have that
\begin{equation}\label{JU:tp:2:CONR}
\{ x\in\Omega {\mbox{ s.t. }}\nabla u(x)\ne0\}\ne\varnothing.
\end{equation}
As a consequence, 
we can take an arbitrary point~$x_0\in\Omega$
such that~$\nabla u(x_0)\ne0$. We also let~$L(x_0)$
be the 
connected component of the
level set of~$u$ in~$\Omega$
passing through~$x_0$.

Furthermore, we notice that assumption~\eqref{spero} in
Lemma~\ref{lemn2}
is satisfied, thanks to the geometric observation in
formulas~\eqref{eqfund} and~\eqref{trentuno}. 
Therefore, we can use Lemma~\ref{lemn2} to say that
\begin{equation}\label{9ikHAKAKKAK67890987654}
L(x_0)=\{ \omega\cdot(x-x_0)=0\}\cap\Omega,\end{equation}
for a suitable~$\omega\in S^{n-1}$, possibly depending
on~$x_0$.

Now we take a vector~$\varpi$ orthogonal to~$\omega$
(of course, $\varpi$ may also depend
on~$x_0$), and 
we consider the straight line
$$ \{ x_0+ \varpi t, \; t\in\R\}.$$
Since the domain~$\Omega$ is bounded, such a line must
intersect~$\partial\Omega$, that is
there exists~$t_0$
such that~$x_0+\varpi t_0\in \partial\Omega$.
Accordingly, by~\eqref{JU:tp},
\begin{equation}\label{Kj7855:A}
u\big(x_0+\varpi t_0)=c.\end{equation}
On the other hand, \eqref{9ikHAKAKKAK67890987654} implies that
\begin{equation*}
u\big(x_0+\varpi t_0\big)=u\big(x_0\big).\end{equation*}
This and~\eqref{Kj7855:A} give that~$u(x_0)=c$.

Notice that this holds for any point~$x_0\in\Omega$
such that~$\nabla u(x_0)\ne0$.
Therefore, we have established that
$$ {\mbox{$u(x)=c$ for any $x\in\Omega\cap
\{ \nabla u\ne0\}$.}}$$
Since the above identity also holds on~$\partial\Omega$
and since~$u$ is constant in each component of~$
\Omega\cap
\{ \nabla u=0\}$, we obtain
that~$u(x)=c$ for any $x\in\Omega$, and
so~$\nabla u$ vanishes identically in~$\Omega$.

This gives a contradiction with~\eqref{JU:tp:2:CONR}, 
thus proving~\eqref{JU:tp:2}.
Hence, the proof of
Theorem~\ref{main1} is completed.
\end{proof}

\section{Proof of Theorem~\ref{main2}}\label{sec3}

We argue towards a contradiction, by supposing that~$u$ is stable.
So, we can use Corollary~\ref{COR:P}:
in particular, we choose~$\varphi\equiv 1$ in~\eqref{prepoi2}, and we 
obtain that
\begin{equation*}\begin{split}
&\int_{\Omega}
\Big[|\nabla_T|\nabla u||^2
+|\nabla u|^2 k_1^2\Big]\,dx
\\&\qquad
+\int_{\partial\Omega} \Big( f(u)\,\partial_\nu u
- \left\langle \nabla u, \partial_{\nu}
\nabla u \right\rangle
-\alpha \,u\,\Delta u-\alpha \,|\nabla u|^2\Big)\,d\sigma
\\
\leq\;& \int_{\Omega}|\nabla u|^2\left\langle A
\nabla\varphi,\nabla\varphi\right\rangle\, dx=0
.\end{split}\end{equation*}
Thus, recalling the Robin boundary condition
\begin{equation}\label{sjgbgj}
\partial_\nu u + \alpha u =0
\quad {\mbox{ on }} \partial\Omega,\end{equation}
we get
\begin{equation}\begin{split}\label{sgrelhhle}
&\int_{\Omega}
\Big[|\nabla_T|\nabla u||^2
+|\nabla u|^2 k_1^2\Big]\,dx
\\&\qquad
-\int_{\partial\Omega} \Big(\alpha\, f(u)\, u
+ \left\langle \nabla u, \partial_{\nu}
\nabla u \right\rangle
+\alpha \,u\,\Delta u+\alpha \,|\nabla u|^2\Big)\,d\sigma
\leq 0
.\end{split}\end{equation}
Moreover, since~$u\in C^2(\overline{\Omega})$, the equation
$$ \Delta u +f(u)=0$$
holds up to the boundary of~$\Omega$, and so~\eqref{sgrelhhle}
becomes
\begin{equation}\label{dslgrhrt554}
\int_{\Omega}
\Big[|\nabla_T|\nabla u||^2
+|\nabla u|^2 k_1^2\Big]\,dx
-\int_{\partial\Omega} \Big(
\left\langle \nabla u, \partial_{\nu}
\nabla u \right\rangle
+\alpha \,|\nabla u|^2\Big)\,d\sigma
\leq 0.\end{equation}
Also, we observe that
$$ \int_{\Omega}
\Big[|\nabla_T|\nabla u||^2
+|\nabla u|^2 k_1^2\Big]\,dx\ge 0.$$
This and~\eqref{dslgrhrt554} give that
\begin{equation}\label{contr}
\int_{\partial\Omega} \Big(
\left\langle \nabla u, \partial_{\nu}
\nabla u \right\rangle
+\alpha \,|\nabla u|^2\Big)\,d\sigma\ge 0. \end{equation}

Now, we
suppose that the boundary~$\partial\Omega$
is represented by the curve~$s\mapsto x(s):= (x_1(s), x_2(s))$,
with~$s\in[0,\ell]$, 
being~$s$ the arc-length.
Therefore, in a neighborhood
of the boundary, a point~$x\in\Omega$ can be written as
$$ x(s,t) = x(s) - t\nu(s),$$
where~$(s,t)$ are the so-called normal coordinates.

In this setting, we recall formulas~(2.2) and~(2.5)
in~\cite{punzo}:
\begin{eqnarray*}
|\nabla u|^2 &=& u_s^2 + u_t^2\\
{\mbox{and }} \quad \left\langle \nabla u, \partial_{\nu}
\nabla u \right\rangle &=& -(\alpha +\kappa) u_s^2
-\kappa\,\alpha^2u^2 +\alpha\, u\, u_{ss} + \alpha\, f(u)\, u,
\end{eqnarray*}
being~$\kappa$ the curvature of~$\partial\Omega$.
Here, $u_s$, $u_t$, $u_{ss}$ and~$u_{tt}$ denote
the first and the second derivatives with respect to the
normal coordinates. 

Also, the boundary condition in~\eqref{sjgbgj} reads as
$$ u_t=\alpha u.$$
Finally, the following formula holds:
$$ u_{tt} -\kappa\,u_t +u_{ss}+f(u)=0\quad {\mbox{ on }}
\partial\Omega,$$
see e.g. the formula above~(2.5) in~\cite{punzo}.

These observations imply that
\begin{eqnarray*}
&&
\left\langle \nabla u, \partial_{\nu}
\nabla u \right\rangle
+\alpha \,|\nabla u|^2\\
&=&
-(\alpha +\kappa) u_s^2
-\kappa\,\alpha^2 u^2 +\alpha\, u\, u_{ss} + \alpha\, f(u)\, u
+ \alpha \left(u_s^2 + u_t^2\right)\\
&=&
- \kappa\, u_s^2
-\kappa\,\alpha^2 u^2 +\alpha\, u\, u_{ss} +\alpha\, f(u)\, u +
\alpha\, u_t^2\\
&=& - \kappa\, u_s^2
-\kappa\,\alpha^2 u^2 +\alpha\, u\, u_{ss} +\alpha\, f(u)\, u +
\alpha^3\, u^2.
\end{eqnarray*}
Hence, integrating over~$\partial\Omega$,
\begin{equation}\begin{split}\label{puuykj}
&\int_{\partial\Omega}\Big( \left\langle \nabla u, \partial_{\nu}
\nabla u \right\rangle
+\alpha \,|\nabla u|^2\Big)\,ds\\
=\;& \int_{\partial\Omega}\Big(
- \kappa\, u_s^2
-\kappa\,\alpha^2 u^2 +\alpha\, u\, u_{ss} +\alpha\, f(u)\, u +
\alpha^3\, u^2
\Big)\, ds.
\end{split}\end{equation}
Now, we observe that
$$ \int_{\partial\Omega}  u\, u_{ss} \,ds = 
-\int_{\partial\Omega} u_s^2\,ds,$$
and so, using this into~\eqref{puuykj}, we conclude that
\begin{eqnarray*}
&&\int_{\partial\Omega}\Big( \left\langle \nabla u, \partial_{\nu}
\nabla u \right\rangle
+\alpha \,|\nabla u|^2\Big)\,ds\\
&=& \int_{\partial\Omega}\Big(
- \kappa\, u_s^2
-\kappa\,\alpha^2 u^2 -\alpha\, u_{s}^2 +\alpha\, f(u)\, u +
\alpha^3\, u^2
\Big)\, ds\\
&=&\int_{\partial\Omega}\Big[\alpha^2u^2\Big(
\frac{f(u)}{\alpha u}-\kappa +\alpha\Big)-
(\alpha+\kappa) u_s^2\Big]\,ds\\
&<&0,
\end{eqnarray*}
thanks to~\eqref{vediamo}. This is in contradiction with~\eqref{contr},
and so the proof of Theorem~\ref{main2} is complete.


\begin{thebibliography}{100}

\bibitem{ARR}
Arrieta, J. M., Carvalho, A. N.:
\textit{Spectral convergence and nonlinear dynamics of reaction-diffusion equations
under perturbations of the domain}, 
J. Differential Equations \textbf{199}, no. 1, 143--178 (2004).

\bibitem{punzo} Bandle, C., Mastrolia, P.,
Monticelli, D. D., Punzo, F.:
\textit{On the stability of solutions of semilinear elliptic
equations with Robin boundary conditions on
Riemannian manifolds},
SIAM J. Math. Anal. \textbf{48}, no. 1, 122--151 (2016).

\bibitem{BOLI}
Bolikowski, \L., Gokieli, M., Varchon, N.:
\textit{The Neumann problem in an irregular domain}, 
Interfaces Free Bound. \textbf{12}, no. 4, 443--462 (2010).

\bibitem{CH} Casten, R. G., Holland, C. J.:
\textit{Instability results for reaction diffusion
equations with Neumann boundary conditions},
J. Differential Equations \textbf{27}, no. 2, 266--273 (1978). 

\bibitem{CNP} Cesaroni, A., Novaga, M., Pinamonti, A.:
\textit{One-dimensional symmetry for semilinear equations
with unbounded drift}, Commun. Pure Appl. Anal. 12, no. 5,
2203--2211 (2013).

\bibitem{CNV}  Cesaroni, A., Novaga, M., Valdinoci, E.:
\textit{A simmetry result for the Ornstein-Uhlenbeck operator},
Discrete Contin. Dyn. Syst. \textbf{34}, no. 6, 2451--2467  (2014).

\bibitem{OLI}
De Oliveira, L. A. F., Pereira, A. L., Pereira, M. C.:
\textit{Continuity of attractors for a reaction-diffusion problem with
respect to variations of the domain}, 
Electron. J. Differential Equations \textbf{100}, 18 pp. (2005). 

\bibitem{DI}
Dipierro, S.: \textit{Geometric inequalities and symmetry
results for elliptic systems},
Discrete Contin. Dyn. Syst. \textbf{33}, no. 8, 3473--3496 (2013).

\bibitem{DP1} Dipierro, S., Pinamonti, A.:
\textit{A geometric inequality and a symmetry
result for elliptic systems involving the fractional Laplacian}, 
J. Differential Equations \textbf{255}, no. 1, 85--119. (2013).

\bibitem{DP2} Dipierro, S., Pinamonti, A.:
\textit{Symmetry results for stable and monotone solutions
to fibered systems of PDEs}, Commun. Contemp. Math.
\textbf{17}, no. 4, 1450035, 22 pp. (2015).

\bibitem{DPV} Dipierro, S., Pinamonti, A., Valdinoci, E.:
\textit{Classification of stable solutions for boundary value problems with nonlinear boundary conditions on Riemannian manifolds with nonnegative Ricci curvature}. To appear in Adv. Nonlinear Anal. arXiv:1710.07329.

\bibitem{DSV} Dipierro, S., Soave, N., Valdinoci, E.:
\textit{On stable solutions of boundary reaction-diffusion
equations and applications to nonlocal problems with Neumann data},
Indiana Univ. Math. J. \textbf{67}, no. 1, 429--469 (2018).

\bibitem{DUP} Dupaigne, L.:
Stable solutions of elliptic partial differential equations,
vol. $143$ of
Chapman \& Hall/CRC Monographs and Surveys in Pure and Applied Mathematics,
CH/CRC, Boca Raton, FL (2011).

\bibitem{FarHab} Farina, A.: \textit {Propri\'et\'es
qualitatives de solutions d'\'equations et syst\`emes d'\'equations
non-lin\'eaires},  \newblock
Habilitation \`a diriger des recherches, Paris VI, (2002). 

\bibitem{FMV} Farina, A., Mari, L., Valdinoci, E.:
\textit{Splitting theorems, symmetry results and
overdetermined problems for Riemannian manifolds},
Comm. Partial Differential Equations \textbf{38},
no. 10, 1818--1862 (2013). 

\bibitem{FNP} Farina, A., Novaga, M., Pinamonti, A.:
\textit{Symmetry results for nonlinear elliptic operators with
unbounded drift}, NoDEA Nonlinear Differential Equations Appl.
\textbf{21}, no. 6, 869--883 (2014).

\bibitem{FSV} Farina, A., Sciunzi, B., Valdinoci, E.:
\textit{Bernstein and De Giorgi type problems:
new results via a geometric approach},
Ann. Sc. Norm. Super. Pisa Cl. Sci. (5) \textbf{7}, no. 4,
741--791 (2008).

\bibitem{ASV2} Farina, A., Sciunzi, B., Valdinoci, E.:
\textit{On a Poincar\'e type formula for solutions of
singular and degenerate elliptic equations}, Manuscripta Math.
\textbf{132}, no. 3--4, 335--342 (2010).

\bibitem{fsv2} Farina, A.,  Sire, Y.,  Valdinoci, E.:
\textit{Stable solutions of elliptic equations on
Riemannian manifolds with Euclidean coverings},
Proc. Amer. Math. Soc. \textbf{140}, no. 3, 927--930 (2012). 

\bibitem{fsv1} Farina, A., Sire, Y., Valdinoci, E.:
\textit{ Stable solutions of elliptic equations on
Riemannian manifolds}, J. Geom. Anal.
\textbf{23}, no. 3, 1158--1172 (2013). 

\bibitem{FV12} Farina, A., Valdinoci, E.:
\textit{The state of the art for a conjecture of De Giorgi
and related problems. In: Du, Y., Ishii, H., Lin, W.-Y. (eds.)},
Recent Progress on Reaction Diffusion System and Viscosity
Solutions. Series on Advances in Mathematics for Applied Sciences,
$372$ World Scientific, Singapore (2008).

\bibitem{fazly}
Fazly, M., Ghoussoub, N.: 
\textit{De Giorgi type results for elliptic systems},
Calc. Var. Partial Differential Equations \textbf{47}, no. 3--4,
809--823 (2013).

\bibitem{HAVE}
Hale, J. K., Vegas, J.:
\textit{A nonlinear parabolic equation with varying domain},
Arch. Rational Mech. Anal. \textbf{86}, no. 2, 99--123
(1984).

\bibitem{FP} Ferrari, F., Pinamonti, A.:
\textit{Nonexistence results for semilinear equations
in Carnot groups}, Anal. Geom. Metr. Spaces \textbf{1},
130--146 (2013).

\bibitem{FV1} Ferrari, F., Valdinoci, E.:
\textit{A geometric inequality in the Heisenberg
group and its applications to stable solutions of semilinear problems},
Math. Ann. \textbf{343}, no. 2, 351--370 (2009).

\bibitem{GI} Giusti, E.: \textit{On the equation of surfaces of prescribed mean curvature.
Existence and uniqueness without boundary conditions},
Invent. Math. \textbf{46}, no. 2, 111--137 (1978).

\bibitem{jimbo}
Jimbo, S.: \textit{On a semilinear diffusion equation on a
Riemannian manifold and its stable equilibrium solutions}.
Proc. Japan Acad. Ser. A Math. Sci. \textbf{60}, no. 10, 349--352 (1984).

\bibitem{J1}
Jimbo, S.: \textit{The singularly perturbed domain and the characterization
for the eigenfunctions with Neumann boundary condition},
J. Differential Equations \textbf{77}, no. 2, 322--350 (1989).

\bibitem{J2}
Jimbo, S.: \textit{Singular perturbation of domains and semilinear elliptic equations. III},
Hokkaido Math. J. \textbf{33}, no. 1, 11--45 (2004).

\bibitem{J3}
Jimbo, S., Morita, Y.: \textit{Remarks on the behavior of certain eigenvalues
on a singularly perturbed domain with several thin channels},
Comm. Partial Differential Equations \textbf{17}, no. 3-4, 523--552 (1992).

\bibitem{LU68} Ladyzhenskaya, O., Uraltseva, N.:
Linear and Quasilinear Elliptic Equations, Academic Press,
New York, (1968).

\bibitem{LL} Lieb, H. H., Loss, M.: Analysis,
vol. $14$ of Graduate Studies in Mathematics,
AMS, Providence, RI (1997).

\bibitem{Lieb} Lieberman, G.: \textit{Boundary regularity for
solutions of degenerate elliptic equations}, 
Nonlinear Anal. \textbf{12}, no. 11, 1203--1219 (1988). 

\bibitem{MWW}
Ma, X.-N., Wang, P.-H., Wei, W.:
\textit{Constant mean curvature surfaces and mean curvature flow with non-zero
Neumann boundary conditions on strictly convex domains},
J. Funct. Anal. \textbf{274}, no. 1, 252--277 (2018).

\bibitem{Mat} Matano, H.: \textit{Asymptotic behavior and stability of solutions of semilinear diffusion equations},
Publ. Res. Inst. Math. Sci. \textbf{15}, no. 2, 401--454 (1979).

\bibitem{J4} Morita, Y.,
Jimbo, S.: \textit{Ordinary differential equations (ODEs) on inertial
manifolds for reaction-diffusion systems in a singularly
perturbed domain with several thin channels},
J. Dynam. Differential Equations \textbf{4}, no. 1, 65--93 (1992). 

\bibitem{PV} Pinamonti, A., Valdinoci, E.:
\textit{A geometric inequality for stable solutions
of semilinear elliptic problems in the Engel group},
Ann. Acad. Sci. Fenn. Math. \textbf{37}, no. 2, 357--373 (2012).

\bibitem{sire1}
Sire, Y., Valdinoci, E.:
\textit{Fractional Laplacian phase transitions and boundary
reactions: a geometric inequality and a symmetry result}.
J. Funct. Anal. \textbf{256}, no. 6, 1842--1864 (2009). 

\bibitem{sire2}
Sire, Y., Valdinoci, E.: \textit{Rigidity results for some boundary
quasilinear phase transitions}.
Comm. Partial Differential Equations \textbf{34}, no. 7--9, 765--784
(2009).

\bibitem{SZ1} Sternberg, P., Zumbrun, K.:
\textit{A Poincar\'e inequality with applications to
volume-constrained area-minimizing surfaces},
J. Reine Angew. Math. \textbf{503}, 63--85 (1998).

\bibitem{SZ2} Sternberg, P., Zumbrun, K.:
\textit{Connectivity of phase boundaries in strictly convex domains},
Arch. Ration. Mech. Anal. \textbf{141}, no. 4, 375--400 (1998).

\bibitem{VEGA}
Vegas, J. M.:
\textit{Bifurcations caused by perturbing the domain in an elliptic equation},
J. Differential Equations \textbf{48}, no. 2, 189--226 (1983).
\end{thebibliography}
\end{document}